\documentclass[oneside,a4paper,12pt]{amsart}

\usepackage[utf8]{inputenc}
\usepackage[T1]{fontenc}    

\usepackage{amsmath}
\usepackage{amsfonts}  
\usepackage{amssymb}
\usepackage{amsthm}
\usepackage{amsaddr}
\usepackage{mathtools}
\usepackage{amscd}

\usepackage{enumitem}

\usepackage{charter}

\usepackage[a4paper,hmargin=3cm,vmargin=3.5cm]{geometry}

\usepackage{tikz}
	\usetikzlibrary{cd,positioning,hobby,intersections,arrows,calc}
	\tikzset{>=latex}

\usepackage[backend=biber,style=alphabetic]{biblatex}
\usepackage[bookmarks]{hyperref}
	\hypersetup{colorlinks=true,allcolors=blue}

\newcommand{\R}{\mathbb{R}}
\newcommand{\Z}{\mathbb{Z}}
\newcommand{\A}{\mathcal{A}} 
\newcommand{\pth}[1]{\ensuremath{\mathcal{P}(#1)}} 

\DeclareMathOperator{\tr}{tr}		
\DeclareMathOperator{\ad}{ad}		
\DeclareMathOperator{\id}{id}		
\newcommand{\bndr}[1]{\partial\mkern-1mu{#1}}	
\newcommand{\isom}{\ensuremath{\cong}} 

\newcommand{\ext}{\ensuremath{\mkern.5mu{d}\mkern-.5mu}} 
\newcommand{\cob}{\ensuremath{\mkern1mu{\delta}}} 
\newcommand{\lder}[1]{\ensuremath{\mathcal{L}_{#1}}} 
\newcommand{\inv}[1]{\ensuremath{{#1}^{-1}}} 

\newcommand{\gext}[1]{\ensuremath{\widehat{#1}}} 
\newcommand{\extseq}[3]{\ensuremath{#1 \hookrightarrow \gext{#2} \to #3}} 
\newcommand{\lie}[1]{\ensuremath{\mathfrak{#1}}} 
\newcommand{\aut}[1]{\ensuremath{\mathrm{aut}(#1)}} 

\newcommand{\coh}{\ensuremath{\mathrm{H}}} 
\newcommand{\coc}{\ensuremath{\mathrm{Z}}} 
\newcommand{\coch}{\ensuremath{\mathrm{C}}} 

\newcommand{\smooth}[1]{\ensuremath{\textrm{Map}(#1)}}

\newcommand{\ballgroup}{\ensuremath{B^3_{\flat}G}} 
\newcommand{\basedball}{\ensuremath{B^3_{\flat,c}G}} 
\newcommand{\loopgroup}{\ensuremath{\Omega^3G}} 
\newcommand{\loopgroupmod}{\ensuremath{\mathcal{S}}} 
\newcommand{\spheregroup}{\ensuremath{S^3G}} 

\newcommand{\potgroup}{\ensuremath{\smooth{\ballgroup,S^1}}} 

\theoremstyle{plain}
\newtheorem{thm}[equation]{Theorem}
\newtheorem{lem}[equation]{Lemma}
\newtheorem{prop}[equation]{Proposition}

\theoremstyle{definition}
\newtheorem{defn}[equation]{Definition}

\theoremstyle{remark}
\newtheorem{rem}[equation]{Remark}

\numberwithin{equation}{section}

\title{A $2$-group construction from an extension of the $3$-loop group $\loopgroup$}
\author{Jouko Mickelsson and Ossi Niemimäki}
\address{Department of Mathematics and Statistics, University of Helsinki}
\email{jouko@kth.se, ossi.niemimaki@helsinki.fi}


\date{\today}

\begin{document}

\thispagestyle{empty}

\begin{abstract}
We define a $3$-loop group $\loopgroup$ as a subgroup of smooth maps from a $3$-ball to a Lie group $G$, and then construct a $2$-group based on an automorphic action on the Mickelsson-Faddeev extension of $\loopgroup$. In this we follow the strategy of Murray et al., who earlier described a similar construction in one dimension. The three-dimensional situation presented here is further complicated by the fact that the $3$-loop group extension is not central. 
\end{abstract}

\maketitle

\section{Introduction}

There is a growing interest in generalizing Lie groups and algebras to higher-dimensional objects in the sense of category theory. In particular, such smooth categorical groups would be valuable as new mathematical objects that would aid in (re)defining and building fundamental physics. One driving notion is that of the string group and its geometric realizations, which are closely tied to the concept of loop groups.

A one-dimensional loop group generalization in this vein has been built in \cite{bscs2007}, in which a specific relation between Lie $2$-algebras and Lie $2$-groups was constructed using groups of based paths. This is shown to lead to a geometric realization of the string group. More generally, there is a no go -theorem which states that the $2$-group generalization coming from the direct Lie $2$-algebra construction for a simple Lie group $G$ allows for a smooth structure only in a very specific situation~\cite{bl2004}. The way around this is to construct an infinite-dimensional Lie $2$-group whose Lie $2$-algebra is \emph{equivalent} to the desired Lie $2$-algebra.

This approach of using categorical equivalences instead of isomorphisms seems to be crucial in constructing such higher objects. With the string $2$-group model in mind, the loop group setting can be extended to quasi-periodic paths~\cite{mrw2017}. Yet there is a topological obstruction for building a strict $2$-group model on this group; the remedy is to flatten the paths around the base point in such a way that there is a categorical equivalence to the original $2$-group. This gives a \emph{coherent} $2$-group model for the string group, which -- in contrast to the model in \cite{bscs2007} -- also admits the action of the circle group $S^1$.

Our aim in this paper is simple: following the strategy from \cite{mrw2017} we construct a similar $2$-group in three dimensions by considering the group of maps from the $3$-sphere to a Lie group $G$. To begin with, we study the Mickelsson-Faddeev cocycle and the related Lie group extension~\cite{mic1987}. This cocycle is crucially dependent on defining Lie-algebra valued $1$-forms over the domain -- in particular, the extension is not central.\footnote{In quantum field theory, the cocycle represents the Hamiltonian quantisation anomaly of massless chiral fermions coupled to external gauge potentials.}

On this $3$-loop group extension we impose the action of the group of smooth maps $\ballgroup$ from the $3$-ball to the Lie group $G$ with the additional condition that the maps are flattened on the boundary $S^2$. We show that it is possible to build a crossed module starting with the group of smooth paths $\ballgroup$ in a similar fashion to the one-dimensional loop group described in \cite{mrw2017}. The main difference is that the lifting of the group action requires some deliberation since the extension of the $3$-loop group is not central. This complicates the structure of the crossed module too, and we need to extend the acting group $\ballgroup$ by the group $\potgroup/S^1$. From the automorphic action of $\gext{\ballgroup}$ on the $3$-loop group extension $\gext{\loopgroup}$ and the natural projection between these two groups we gain a crossed module and thus a strict $2$-group. 

One of the strengths of the original one-dimensional loop groupoid construction is that it allows the action of the circle group. In three dimensions one would like to have the corresponding action of $SO(4)$, but there seems to be no straight-forward way to incorporate this into the picture presented here. The extension of the $3$-loop group necessitates a fixed point in $B^3$ -- namely, the contracted boundary $S^2$ -- and this choice cannot be equivariant under the symmetry action.

\subsection{Notation and conventions}

Unless otherwise stated, $G$ is a simply-connected Lie group. The group identity element is denoted by $e$ throughout, and we write $\pth{G}$ for the group of based paths in $G$. The Lie algebra of $G$ is denoted by $\lie{g}$. We identify the $1$-sphere $S^1$ with the circle group.

In places there is an implicit assumption to consider the connected component of a given group in case the connectedness for the whole group is not available.

\section{$3$-loop group and an automorphic action}

Let $\spheregroup$ be the group of smooth maps from the $3$-sphere to a Lie group $G$. Since every $n$-sphere can be given as the quotient $B^n/S^{n-1}$, where the boundary of the $n$-ball is contracted to a point (call it the north pole of the sphere $S^n$), we define by analogy the flattened group 
\[
	\ballgroup = \{ f \in \smooth{B^3,G} : \partial_r f = 0 \textrm{ on the boundary } S^2 \} ,
\]
where the group multiplication is as usual the point-wise multiplication in the domain. We then have the \emph{$3$-loop group} defined as
\[
	\loopgroup = \{ f \in \ballgroup : f \textrm{ extends to } S^3 \textrm{ and } f(S^2) = e \} .
\]
Note that while the group of rotations $SO(4)$ acts on the sphere group $S^3G$, it does not act on $\loopgroup$ since the boundary of $B^3$ is the contracted fixed point and the action cannot be extended to $S^3$.

It is easy to see that $\ballgroup$ retains the Lie group structure of $\smooth{B^3,G}$, and likewise for $\loopgroup$. In the case of one-dimensional loop group $\Omega G$ and the free loop group $LG$, there is the split exact sequence
\[
	\Omega G \hookrightarrow LG \to G .
\]
For the $3$-loop group this relationship is retained as follows. If we denote by $\basedball$ the subgroup of $\ballgroup$ of maps that are constant on the boundary $S^2$, the group $\basedball$ is in fact a principal $\loopgroup$-bundle over $G$:
\[
	\loopgroup \hookrightarrow \basedball \xrightarrow{\phi} G ,
\]
where $\inv{\phi}(e) = \loopgroup$ is the canonical fibre. 

Furthermore, we note that $\loopgroup$ is a split normal subgroup of $\ballgroup$. The quotient map
\[
	q: \ballgroup \to \ballgroup/\loopgroup \isom \smooth{S^2,G}
\]
has $\loopgroup$ as its kernel.

\subsection{Abelian extension $\gext{\loopgroup}$ and an automorphic action on $\loopgroup$}

Given any mapping group $X^G = \smooth{X,G}$ for some $n$-dimensional Lie group $X$, there is an extension by the Abelian group of smooth maps $\smooth{X^G,S^1}$~\cite[pp. 66--67]{ps1986}. Let us then consider an Abelian extension~\cite{mic1987}
\[
	\extseq{\potgroup}{\spheregroup}{\spheregroup} .
\]
In the gauge-theoretical formulation the fibre is the space of vector potentials $\A$ on $S^3$ with values in the Lie algebra $\lie{g}$. The group $\potgroup$ is closely related since the potentials $A\in \A$ can be identified with functions $f \in \ballgroup$ in the sense of Maurer-Cartan forms so that $A \coloneqq \inv{f} \ext f$.\footnote{See the motivation for the original idea in \cite{mic1987} or \cite{mw2016}.} 

On the Lie algebra level we have the corresponding Mickelsson-Faddeev cocycle
\begin{equation}\label{eq:s3g-cocycle}
	\theta(A;x,y) = \frac{i}{24\pi}\int_{S^3} \tr A [\ext x, \ext y] .
\end{equation}
No modification is needed to write this in terms of the $3$-loop group:
\[
	\extseq{\potgroup}{\loopgroup}{\loopgroup} .
\]
The corresponding Lie algebra cocycle is essentially of the same form as above.

One can verify by direct but tedious computation (see Appendix \ref{app1}) that the right adjoint action $(\cdot)^f$ with $f \in \ballgroup$ adds only a $2$-coboundary term to the Lie algebra cocycle:
\begin{equation}\label{eq:invariant_2cocycle}
	\theta^f = \theta + \{\textrm{coboundary}\} .
\end{equation}
Hence the cocycle is cohomologically invariant under such automorphisms.

More generally, we have the following proposition.
\begin{prop}
\label{p:path_group_invariance}
Let $K$ be a simply-connected Lie group such that $K = H/N$, where $N \subset H$ is a normal subgroup of a simply-connected Lie group $H$ for which the first three integral cohomology groups are trivial. Denote by $\pi: H \to K$ the natural projection so that the fibration has $N$ as the canonical fibre. Then, given a cohomology class $[\omega] \in \coh^3(K,\Z)$, there is a $2$-form $\theta$ on $H$ given by the pullback $\pi^*\omega = \ext \theta$ such that the following holds:
\begin{enumerate}
	\item $\theta$ is closed in vertical directions and defines a $2$-cocycle on the Lie algebra \lie{n} of $N$ with values in $\smooth{H,\R}$.
	\item $[\theta]$ in $\coh^2({\lie{n},\smooth{H,\R}})$ is uniquely determined by $\omega$.
	\item $[\theta]$ in $\coh^2({\lie{n},\smooth{H,\R}})$ is invariant with respect to right action $r_f^*$, where $f\in H$.
\end{enumerate}
\begin{proof} 
Consider $\theta' = \theta + \psi$, where $\psi$ is a closed $2$-form. Since $\coh^2(H,\Z)$ is trivial, the $2$-form $\psi$ is also exact. Hence the $2$-cocycle $\theta'-\theta$ in $N$ is a coboundary, and $\theta$ defines a $2$-cocycle on the Lie algebra of $N$ with values in $\smooth{H,\R}$.

Let $f \in H$ and consider a path $f_t$ in $H$, where $t\in[0,1]$ with $f(0) = e$ and $f(1) = f$. Then the cohomology class $[\omega]$ is invariant by the right action in $K$ by $\pi(f)$:
\[
	r^{*}_{\pi(f)}\omega \sim \omega .
\]
Moreover,
\[
	\ext (r^{*}_f \theta ) = r^{*}_f \ext \theta = r^{*}_{\pi(f)} \omega \sim \omega .
\]
Now, given $\omega' = \omega + \ext \phi$ for some $2$-form $\phi$, we have the pullback
\[
	\pi^* (\omega') = \pi^* \omega + \ext (\pi^* \phi) = \ext \theta  + \ext (\pi^* \phi) ,
\]
and hence, since the second cohomology group for $H$ is trivial, we can write for some $1$-form $\alpha$
\[
	\theta' = \theta + \pi^* \phi + \ext \alpha ,
\]
where the term $ \pi^* \phi$ is zero in vertical directions. Thus the Lie algebra cohomology class $[\theta]$ is uniquely determined by $\omega$ and it is invariant with respect to $r^*_f$ for any $f\in H$.
\end{proof}
\end{prop}

\begin{rem}\label{r:path_group} 
The proposition can be applied to $3$-loops with $K=S^2G$, $N=S_e^3G$ and $H=\pth{S^2G}$. From this we can squeeze out a general formula for computing the coboundary (see Appendix~\ref{app1}).
\end{rem}

\begin{rem}Since $[\omega] \in \coh^3(H/N,\Z)$, one can think of $\omega$ as a representative of the Dixmier-Douady class on $K$. In the gauge-theoretic framework the group $K$ models the moduli space of based gauge transformations which parametrisizes fermionic Fock spaces~\cite{singer1981}. There is a natural principal bundle
\[
	\mathcal{G}_e \hookrightarrow \A \to \A/\mathcal{G}_e ,
\]
where $\A$ is the space of gauge potentials, and $\mathcal{G}_e$ is the space of based gauge transformations related to the symmetry group $G$. The Dixmier-Douady class in $\coh^3(\A/\mathcal{G}_e)$ represents the obstruction to lifting the $\mathcal{G}_e$-bundle to its extension by $\smooth{\A,S^1}$. This is also the origin of the Lie algebra cocyle \eqref{eq:s3g-cocycle}. \cite{cmm1997} 
\end{rem}

We now lift this action invariance of the Lie algebra cocycle to the Lie group level. To this end we need further results concerning cohomology; more details on the following can be found in \cite{neeb2004} and \cite{mw2016}.

Let $H$ be a Lie group. Recall that an Abelian Lie group $A$ is a smooth $H$-module, if it is a $H$-module and the action map $H\times A \to A$ is smooth. Assuming that $N$ is a normal subgroup of $H$, the group of $N$-invariant elements of $A$ is denoted by
\[
	A^N = \{ a \in A : (\forall n \in N) ~ n.a = a \} ,
\]
which is a $H$-submodule of $A$. 

\begin{defn}[Refined cohomology~{\cite[Appendix D]{neeb2004}}]
We denote \emph{smooth group cohomology} by $\coh^n_s(N,A)$ and \emph{continuous Lie algebra cohomology} by $\coh^n_c(\lie{n},\lie{a})$. The gist is that the maps $N^n \to A$ (resp. $\lie{n}^n \to \lie{a}$) are smooth (resp. continuous). The smoothness and continuity are defined locally in a neighborhood of the identity.

A cocycle $f \in \coc^n_s(N,A)$ is \emph{smoothly cohomologically invariant with respect to $H$} if there is a map
\[
	\phi: H \to \coch_s^{p-1}(N,A) \quad \text{such that} \quad  \ext(\phi(h)) = h.f - f \quad \forall h \in H ,
\]
and the map
\[
	H \times N^{p-1} \to A : (h,n_1,\dots,n_{p-1}) \mapsto (h.f-f)(n_1,\dots,n_{p-1})
\]
is smooth in an identity neighborhood. This gives us \emph{smoothly invariant cohomology classes of $N$ with values in $A$}.
\end{defn}

\begin{thm}[Cohomology homomorphism~{\cite[Thm. VII.2]{neeb2004}}]\label{th:cohom_hom}
Let $N$ be a connected Lie group and $A \isom \lie{a}/\Gamma_A$ a smooth $N$-module, where $\Gamma_A \subset \lie{a}$ is a discrete subgroup of the sequentially complete locally convex space $\lie{a}$. Then there is an exact sequence
\[
	\hom(\pi_1(N),A^N) \to \coh^2_s(N,A) \to \coh^2_c(\lie{n},\lie{a}) .
\]
\end{thm}

Consider now the right adjoint action of $\ballgroup$. In order to lift this action we need to impose $\pi_1(\loopgroup) = 0$. In general, it is known that 
\[
	\pi_k(S^n_eG) \isom \pi_{k+n}(G) .
\]
For instance, if we pick $G = SU(p)$ for $p>2$, the group $\loopgroup$ is simply connected; for $p=2$ we have topological sectors corresponding to $\Z_2$ -- we discuss this specific case separately below. Furthermore, we may restrict the action of $\ballgroup$ to the \emph{connected component} of $\loopgroup$ in case non-connectedness becomes an issue.

Assuming $\pi_4(G)$ is trivial, $H = \ballgroup$, $N = \loopgroup$ and $A = \potgroup$ as above, the sequence of Theorem~\ref{th:cohom_hom} gives a monomorphism $\coh^2_s(N,A) \to \coh^2_c(\lie{n},\lie{a})$; the Lie algebra $\lie{a}$ of $A$ is of the form $\smooth{\ballgroup,i\R}$, with the choice of $\Gamma_A = \Z$. In light of Equation~\eqref{eq:invariant_2cocycle}, the Lie algebra cocycle is smoothly cohomologically invariant with respect to right adjoing action by $f\in\ballgroup$, and  by above monomorphism the same holds on the group level.

It is then natural to ask whether we can lift the action of $\ballgroup$ on $\loopgroup$ to its extension $\gext{\loopgroup}$. Given any element $f \in\ballgroup$, we know that there is a smooth action by conjugation both on $\loopgroup$ and the fibre $\potgroup$. However, a map from $\ballgroup$ to $\aut{\gext{\loopgroup}}$ obtained in this fashion is not a group homomorphism; this can be seen by a direct computation already on the Lie algebra level. Generally, we have the following proposition.
\begin{prop}[Lifting homomorphism~{\cite[Prop. 3.8]{mw2016}}]\label{p:aut_lift}
Let $H$ be a Lie group, $N$ a connected normal Lie subgroup of $H$, $\theta \in \coc^2_s(N,A)$ a smooth $2$-cocycle and $\gext{N}$ the corresponding Abelian Lie group extension. Then the smooth group homomorphism $\psi: H \to \aut{A} \times \aut{N}$ lifts to a smooth homomorphism $\gext{\psi}:H\to \aut{\gext{N},A}$ if and only if 
\begin{enumerate}
	\item $\theta$ is smoothly cohomologically invariant with respect to $H$, and 
	\item the corresponding cohomology class $[\ext_{\psi}\phi] \in \coh^2_s(H,\coc^1_s(N,A))$ is trivial, where the $1$-cocycle $\phi$ is defined via $\ext_N (\phi(h)) = h.\theta - \theta$ for any $h \in H$.
\end{enumerate}
\end{prop}

The first condition is fulfilled in the case of the $3$-loop group. However, the cohomology class given by $\ext_\psi \phi$ is not trivial, and the action of $\ballgroup$ does not lift to the extension. However, keeping in mind that our goal is a crossed module, we can reroute our approach through central extensions as will be explained in the next section. Let us conclude by gathering the results of this section in the following proposition:
\begin{prop}
\label{p:invariance}Let $\pi_1(\loopgroup)$ be trivial. Then the $2$-cocycle $\theta$ of the Abelian Lie group extension
\[
	\extseq{\potgroup}{\loopgroup}{\loopgroup}
\]
is smoothly cohomologically invariant with respect to the right adjoint action of $\ballgroup$. However, the action of $\ballgroup$ on $\loopgroup$ does not lift to an automorphic action on $\gext{\loopgroup}$.
\end{prop}

\subsection{The case of $SU(2)$}
For $G=SU(2) \isom S^3$ we have $\pi_4(G) = \Z_2$. Moreover, in this case the above Lie algebra cocycle is identically zero~\cite{mic1987}. There is however a non-trivial extension
\[
	\Z_2 \hookrightarrow \mathcal{G} \to \loopgroup ,
\]
where $\mathcal{G}$ is the group of homotopy classes of paths $f$ in $\loopgroup$ such that $f(0) = e$ and $f(1) = g$, when we consider only the connected component. This is a covering group for $\loopgroup$ and there is a natural projection to path points.

The right conjugation by $f \in \ballgroup$ acting on any given path point does not change the homotopy. Thus the right adjoint action of the group $\ballgroup$ on $\loopgroup$ lifts to an automorphic action on the extension $\mathcal{G}$.

\section{$2$-group from a crossed module}

The aim of this paper to is construct an action groupoid that would fulfill the axioms of a crossed module, and thus define a strict $2$-group~\cite{bs1976}. Let us revisit the definitions.
\begin{defn}[Crossed module]
Let $G$ and $H$ be groups, and consider morphisms
\[ 
	\delta : H \to G \quad \textrm{and} \quad \alpha : G \to \aut{H} .
\]
We say that  $[\delta:H \to G]$ is a crossed module if the following two diagrams commute.
\[
\begin{tikzcd}
H \times H \arrow[rd,"\ad"] \arrow[r,"\delta\times\id"] & G \times H \arrow[d,"\alpha"]\\
	& H
\end{tikzcd}
\qquad
\begin{tikzcd}
G \times H \arrow[d,"\id\times\delta"] \arrow[r,"\alpha"] & H \arrow[d,"\delta"] \\
G \times G \arrow[r,"\ad"] & G
\end{tikzcd}
\]
Equivalently, if we denote by $h^g$ the element-wise action of $G$ on $H$, the diagrams correspond to the equations
\[
	h^{\delta(h')} = \inv{h'} h h'
\]
and
\[
	\delta(h^g) = \inv{g} \delta(h) g 
\]	
for all $h,h' \in H$ and $g \in G$.
\end{defn}

\begin{defn}[Smooth crossed module]
If the groups $G$ and $H$ in a crossed module $[\delta:H\to G]$ are Lie groups and the action defined by the morphism $\alpha$ is smooth, the crossed module is called a Lie crossed module, or a smooth crossed module.
\end{defn}

As stated in Proposition~\ref{p:invariance}, we do not have a lift of the action of $\ballgroup$ to the extension $\gext{\loopgroup}$. Even if we had, this would not be sufficient for constructing a crossed module, since the extension $\gext{\loopgroup}$ would not allow for a natural morphism $\delta$ to the acting group $\ballgroup$. Let us then extend the group $\ballgroup$ as follows. There is a topologically trivial Lie group extension
\[
	\extseq{\potgroup/S^1}{\ballgroup}{\ballgroup} ,
\]
which on the Lie algebra level links back to the Mickelsson-Faddeev cocycle (modulo the Lie algebra of $S^1$). The reason for modifying the center with $S^1$ comes from the need to fulfill the cocycle condition when considering the domain $B^3$ instead of $S^3$: there is a constant boundary term
\[
	\cob \theta = \frac{i}{24\pi} \int_{\bndr B^3} \tr x [\ext y, \ext z] - y [\ext z, \ext x] + z [\ext x, \ext y] ,
\]
which then vanishes by construction since the Lie algebra of $S^1$ is $i\R$.

Note that now the right adjoint action of $\gext{\ballgroup}$ defines automorphisms on the Lie algebra of the extension $\gext{\loopgroup}$ so that we have a homomorphism
\[
	\gext{\ballgroup} \to \aut{\gext{\Omega^3\lie{g}}} ,
\]
which leads to the following proposition.
\begin{prop}
If the (connected component of the) group $\gext{\loopgroup}$ is simply connected, then the homomorphism
\[
	\alpha' : \gext{\ballgroup} \to \aut{\gext{\Omega^3\lie{g}}}
\]
lifts to a homomorphism
\[
	\alpha : \gext{\ballgroup} \to \aut{\gext{\loopgroup}} 
\]
such that the action of $\aut{\gext{\loopgroup}}$ on $\gext{\loopgroup}$ is smooth.
\end{prop}
\begin{proof}
The statement is a direct consequence of Corollary 9.5.11 in \cite[p. 341]{hn2012}, which gives a Lie group isomorphism between the automorphism groups of a connected and simply-connected Lie group and its Lie algebra. Hence the following diagram commutes.
\[
	\begin{tikzcd} & \aut{\gext{\loopgroup}} \arrow{d}[sloped,above]{\isom} \\
		\gext{\ballgroup} \arrow[ru,"\alpha"] \arrow[r,"\alpha'"] & \aut{\gext{\Omega^3\lie{g}}}
	\end{tikzcd}
\]
\end{proof}

\begin{rem}
For the canonical central extension of the one-dimensional loop group it holds that $\pi_1(\gext{\Omega G}) = 0$. A similar argument can be used in the case of $\loopgroup$.
\end{rem}

On the other hand, whether $\gext{\loopgroup}$ is simply connected or not, we can consider the group
\[
	\loopgroupmod = \potgroup/S^1 \rtimes \loopgroup
\]
which is a normal subgroup of $\gext{\ballgroup}$, and hence there is a smooth right adjoint action of $\gext{\ballgroup}$ on $\loopgroupmod$ by conjugation. Furthermore, while the previously introduced $3$-loop group extension is not central, there is a central extension
\begin{equation}\label{eq:central_ext}
	\extseq{S^1}{\loopgroupmod}{\loopgroupmod} ,
\end{equation}
where the group $\gext{\ballgroup}$ acts trivially on the central $S^1$. Note that this is none other than the original extension $\gext{\loopgroup}$ of the $3$-loop group.

Now we can consider the right adjoint action of $\gext{\ballgroup}$ on the group $\gext{\loopgroup}$. The cocycle corresponding to $\gext{\loopgroup}$ is the familiar $2$-cocycle $\theta$, and so by Proposition~\eqref{p:invariance} and the construction above we know that it is smoothly cohomologically invariant with respect to the action of $\gext{\ballgroup}$. Hence the first condition of Proposition~\ref{p:aut_lift} for lifting the homomorphism
\[
	\psi: \gext{\ballgroup} \to \aut{S^1} \times \aut{\potgroup/S^1 \rtimes \loopgroup}
\]
to a homomorphism
\[
	\gext{\psi} : \gext{\ballgroup} \to \aut{\gext{\loopgroup},\potgroup}
\]
is fulfilled.

For the second condition we observe the following. Let groups $H, N$ and $A$ be as in Proposition~\ref{p:aut_lift} with the cocycle $\theta \in \coc^2_s(N,A)$ corresponding to the Abelian extension $\gext{N}$. Assuming that $N$ is split normal subgroup of $H$ and that $A^N$ is a split Lie subgroup of $A$, there is a commuting diagram~\cite[Lemma 4.3 and the preceeding definitions]{mw2016}
\[
\begin{tikzcd}
	1 \arrow[r] & A \arrow[d] \arrow[r] & \gext{N} \arrow[d] \arrow[r] & N \arrow[d] \arrow[r] & 1 \\
	1 \arrow[r] & \coc_s^1(N,A) \arrow[r] & \Gamma \arrow[r] & H \arrow[r] & 1   
\end{tikzcd}
\]
where $\Gamma$ is the extension of $H$ by $Z_s^1(N,A)$. Furthermore, there is a smooth automorphic action $\Gamma \to \aut{\gext{N}}$; the pair $(\gext{N}, \Gamma)$ in fact has the structure of a smooth crossed module provided that the cohomology group $\coh^1_s(N,A)$ is trivial~\cite[Prop. 4.5 and Thm. 4.6]{mw2016}. In particular, this condition ensures that $\Gamma$ is a Lie group extension and that $Z_s^1(N,A) \isom A/A^N$.

\begin{lem}
Let $G=SU(p)$ with $p\geq 3$. Then the group $\coh^1_s(\loopgroup,\potgroup)$ is trivial. 
\end{lem}
\begin{proof}
Denote $N = \loopgroup$, $A=\potgroup$ and $H=\smooth{S^2,G}$; note that we can identify $H$ with $\ballgroup/N$. By the Bott periodicity the cohomology of $H$ in low degrees is generated by elements $\alpha_{2k+1}$ in odd degrees $(k=1,2, \dots)$. In particular, $\coh^2_s(H, \Z)$ vanishes for $p \geq 3$ and there are no nontrivial circle bundles over $H$.

On the other hand, an element $c_1\in \coh^1_s(N, A)$ describes a circle bundle $Q$ over $H = \ballgroup/N$: elements in $Q$ are equivalence classes of pairs $(g, \lambda) \in \ballgroup \times S^1$ with the equivalence relation $(g, \lambda) \sim (gu, c_1(g;u)\lambda)$. The bundle can be trivialized if and only if $c_1(g; u)= f(gu)f(g)^{-1}$ for some $f: \ballgroup \to S^1$.  Thus the group $H_s^1(N,A)$ is trivial. 
\end{proof}

Assuming that we have the conditions satisfying the preceeding Lemma there is then a smooth homomorphism
\[
	\Gamma \to \aut{\gext{\loopgroup}} ,
\]
where the Lie group $\Gamma$ is defined by the extension
\[
	\potgroup/{\smooth{S^2G,S^1}} \hookrightarrow \Gamma \to \ballgroup .
\]
Here we have identified $A^N \coloneqq \potgroup^{\loopgroup}$ with the group ${\smooth{S^2G,S^1}}$.

We further note that there is a natural smooth homomorphism
\[
	p: \potgroup/S^1 \to \potgroup/{\smooth{S^2G,S^1}} ,
\]
and that the fibre quotient ${\smooth{S^2G,S^1}}$ acts trivially on $\gext{\loopgroup}$. Hence the map $\Gamma \to \aut{\gext{\loopgroup}}$ lifts to a smooth homomorphism
\[
	\alpha: \gext{\ballgroup} \to \aut{\gext{\loopgroup}} ,
\]
which gives us a unique action
\[
	\gext{\loopgroup} \times \gext{\ballgroup} \to \gext{\loopgroup} .
\]
This defines an action groupoid $\gext{\loopgroup}//\gext{\ballgroup}$, 
which can be in fact extended to a crossed module.

By the central extension \eqref{eq:central_ext}, we have an epimorphism
\[
	\gext{\loopgroup} \to \potgroup/S^1 \rtimes \loopgroup ,
\]
where $\potgroup/S^1 \rtimes \loopgroup $ is a normal subgroup of $\gext{\ballgroup}$. Now the above action groupoid fulfills the requirements of a crossed module by construction. Let $\alpha: \gext{\ballgroup} \to \aut{\gext{\loopgroup}}$ as above. The morphism $\delta: \gext{\loopgroup} \to \gext{\ballgroup}$ is gathered as a composite of the above epimorphism and the natural inclusion:
\[
	\gext{\loopgroup} \to \potgroup/S^1 \rtimes \loopgroup  \hookrightarrow \gext{\ballgroup} .
\]

This is the three-dimensional strict $2$-group analogue to that in \cite[Definition 3.5.]{mrw2017}. In contrast, we have no comparable equivalence of Lie groupoids (\cite[Theorem 3.6.]{mrw2017}), since there is no analogue to the non-smooth quasi-periodic action. Likewise, there seems to be no way to include the $SO(4)$-equi\-variance in this picture.


\printbibliography

\newpage
\appendix

\section{}
\label{app1}

\subsection{Explicit cohomology invariance of the $2$-cocycle}

Consider the following $2$-cocycle for the Lie algebra $\Omega^3\lie{g}$:
\[
	\theta(A;x,y) = \int_{S^3} \tr A [\ext x, \ext y] .
\]
We wish to show that the cocycle is cohomologically invariant under the right adjoint action of the group $\ballgroup$.

On the Lie algebra the action is the right conjugation by $f \in \ballgroup$: 
\begin{align*}
	A \mapsto \inv{f} A f + \inv{f} \ext{f} , \quad x \mapsto \inv{f} x f ,
\end{align*}
so that
\[
	\\ 
    \ext (x^f) = -\inv{f}\omega x f + \inv{f} \ext x f + \inv{f} x \ext f ,
\]
where $\omega \coloneqq \ext{f}\inv{f}$ is the right-invariant Maurer-Cartan form. We then get the conjugated cocycle $\theta^f$:
\begin{align*}
  \int_{S^3}\tr A & [\ext x, \ext y] - A \omega [x,y] \omega + A(x\omega\ext y - y \omega \ext x - x\omega^2y + y\omega^2x - \ext x \omega y + \ext y \omega x) \\
  	  & + (A\omega + \omega A)(x\omega y - y \omega x) - A \omega(x\ext y - y \ext x) + \omega A(\ext x y - \ext y x) \\
  	  & + 2\omega^2(x\omega y - y\omega x) - \omega(\ext x \omega y - \ext y \omega x + x\omega^2 y - y \omega^2 x - x\omega \ext y + y \omega \ext x) \\ 
  	  & +\omega^2(\ext x y - \ext y x - x \ext y - y \ext x) + \omega[\ext x, \ext y] - \omega^3[x,y] .
\end{align*}
Using $\omega^2 = \ext \omega$ and Stokes' theorem, we can reduce the terms independent of $A$ to:
\[
	\omega^2(x\omega y - y\omega x) - \omega[\ext x, \ext y] - \omega^3[x,y] .
\]

Let $\lambda(A;z)$ be a $1$-cochain. Then,
\[
	\cob(\lambda)(A;x,y) = \lder{x} \lambda(A;y) - \lder{y} \lambda(A;x) - \lambda(A;[x,y]),
\]
where the derivation $\lder{x}$ acts on $A$ as 
\[
	\lder{x} A = [A,x] + \ext x .
\]
Let $\lambda_i(A;z)$ be $1$-cochains as follows:
\begin{align*}
	&\lambda_1(A;z) = \int_{S^3}\tr A \omega [\omega,z] , \\
	&\lambda_2(A;z) = \int_{S^3}\tr [\omega, A] \ext z  , \\
	&\lambda_3(A;z) = \int_{S^3}\tr \omega^3 z .
\end{align*}
The coboundaries are
\begin{align*}
	\cob\lambda_1(A;x,y) ={}& \int_{S^3}\tr A(x\omega^2 y - y\omega^2 x) - \omega^2(x\ext y - y\ext x) \\
						& \phantom{\int_{S^3}\tr} + A \omega [x,y] \omega  - (A \omega + \omega A)(x\omega y - y \omega x) + \omega (x\omega \ext y - y \omega \ext x) , \\
	\cob\lambda_3(A;x,y) ={}& \int_{S^3}\tr A(\ext x \omega y - \ext y \omega x) - \omega A(\ext x y - \ext y x) \\
						& \phantom{\int_{S^3}\tr} - A(x\omega \ext y - y \omega \ext x) + A \omega(x \ext y - y\ext x) + 2\omega[\ext x, \ext y] , \\
	\cob\lambda_3(A;x,y) ={}& \int_{S^3}\tr \omega^3[x,y] .
\end{align*}
Then clearly
\begin{align*}
	\theta^f = \theta - \sum_i \cob \lambda_i ,
\end{align*}
as wanted.

\subsection{General formula}

Let us expand on Proposition \ref{p:path_group_invariance} and Remark \ref{r:path_group}. Assume $f,g \in \pth{S^2G}$ and consider $2$-cocycle $\omega(g;x,y)$ for the Lie algebra of $\Omega(S^2G)$. We have the right action
\[
	\omega^f(g;x,y) = \omega(gf; \ad_f(x), \ad_f(y) ) ,
\]
where the adjoint action is $\ad_f(x) = \inv{f}xf$ in a matrix representation. We wish to show that $\omega^f - \omega$ is a coboundary. The argument is essentially the same as in the Poincaré lemma.

Let $f_t$ be a path in $\pth{S^2G}$ with $f_0 = e$ and $f_1 = f$. Let us then rewrite the sought term as follows:
\begin{align*}
	\omega^f - \omega &= \int_0^1 \frac{d}{dt} \omega^f  \, \ext t \\
		&= \int_0^1 \omega (g\dot{f_t}; \ad_{f_t}{x},\ad_{f_t}y) \, \ext t  + \int_0^1 \omega (gf_t;[\inv{f_t} x f_t , \inv{f_t} \dot{f_t} ] , \inv{f_t} y f_t) \, \ext t \\
		&\phantom{=} + \int_0^1 \omega (gf_t; \inv{f_t} x f_t , [\inv{f_t} y f_t , \inv{f_t} \dot{f_t} ] ) \, \ext t \\
		&= - \int_0^1 \omega (gf_t; \inv{f_t} \dot{f_t}, [\inv{f_t} x f_t, \inv{f_t} y f_t ] ) \, \ext t \\
		&\phantom{=} + \int_0^1 \lder{\inv{f_t}x f_t} \omega (gf_t; \inv{f_t} \dot{f_t}, \inv{f_t} y f_t ) \, \ext t \\
		&\phantom{=} + \int_0^1 \lder{\inv{f_t}y f_t} \omega (gf_t; \inv{f_t} \dot{f_t}, \inv{f_t} x f_t ) \, \ext t , \\
\end{align*}
where on the last equality we have used the $2$-cocycle property of $\omega$. The result indeed is a coboundary $(\cob \theta)(g;x,y)$, where
\[
	\theta(g;z) = \int_0^1 \omega(gf_t; \inv{f_t} \dot{f_t}, \inv{f_t} z f_t) \, \ext t . 
\]
It is important that either $x$ or $y$ must be periodic (hence $\Omega(S^2G)$ so that both are), for this ensures that the boundary terms vanish.

Let us sketch this more concretely for the one-dimensional loops with the known $2$-cocycle
\[
	\omega(x,y) = \frac{1}{2\pi}\int_0^{2\pi} \tr x y' \, \ext \phi 
\]
in the Lie algebra $\Omega{\lie{g}}$. Here we have abbreviated $\frac{d}{d\phi} \alpha = \alpha'$ in contrast to the path-related derivation $\frac{d}{dt} \alpha = \dot{\alpha}$. Now,
\begin{align*}
	\theta(x) &= \int_0^1 \omega(\inv{f}_t \dot{f_t}, \inv{f}_t x f_t )  \, \ext t \\
			 &= \frac{1}{2\pi}\ \int_0^1 \int_0^{2\pi}  \tr \inv{f}_t \dot{f_t} \frac{d}{d\phi} (\inv{f}_t x f_t ) \, \ext \phi \, \ext t .
\end{align*}
We can choose a path $f_t = e^{tz(\phi)}$, where $z:[0,2\pi] \to \lie{g}$. We then have
\[
	\inv{f}_t \dot{f_t} = z(\phi) ,
\]
and so for the cochain
\begin{align*}
	\theta(x) &= \frac{1}{2\pi} \int_0^1 \int_0^{2\pi}  \tr z(\phi) \frac{d}{d\phi} (\inv{f}_t x f_t ) \, \ext \phi \, \ext t \\
			&= - \frac{1}{2\pi} \int_0^1 \int_0^{2\pi}  \tr z'(\phi) (\inv{f}_t x f_t ) \, \ext \phi \, \ext t \\
			&= - \frac{1}{2\pi} \int_0^1 \int_0^{2\pi}  \tr f_t z'(\phi) \inv{f}_t x   \, \ext \phi \, \ext t
\end{align*}
by partial integration and the cyclicity of the trace. Next we note that
\begin{align*}
	\frac{d}{dt} (f_t \frac{d}{d\phi} \inv{f}_t) &= f_t z(\phi) \frac{d}{d\phi} \inv{f}_t - f_t \frac{d}{d\phi} (z(\phi) \inv{f}_t) = -f_t z'(\phi) \inv{f}_t .
\end{align*}
Inserting this back to the cochain we have
\begin{align*}
	\theta(x) &= \frac{1}{2\pi} \int_0^1 \int_0^{2\pi}  \tr \frac{d}{dt} (f_t \frac{d}{d\phi} \inv{f}_t) x \, \ext \phi \, \ext t ,
\end{align*}
so that finally by integrating over $t$ we see that our cochain indeed is a coboundary~\cite[Example 3.10]{mw2016}
\[
	\theta(x) = - \frac{1}{2\pi} \int_0^{2\pi}  \tr f'_1 \inv{f}_1 x \, \ext \phi .
\]

\end{document}